\documentclass[11pt, oneside]{amsart}   	
\usepackage[utf8]{inputenc}
\usepackage{geometry}                		
\usepackage[parfill]{parskip}    		
\usepackage{graphicx}				
\usepackage{amsmath}
\usepackage{amssymb}
\usepackage{amsthm}

\usepackage{dsfont}

\usepackage{esint}

\usepackage{mathrsfs}

\usepackage[english]{babel}
\usepackage{fancyhdr}

\theoremstyle{theorem}
\newtheorem{theorem}{Theorem}[section]

\theoremstyle{definition}
\newtheorem{definition}[theorem]{Definition}

\theoremstyle{definition}

\theoremstyle{theorem}
\newtheorem*{theorem*}{Theorem}

\theoremstyle{theorem}
\newtheorem{lemma}[theorem]{Lemma}

\theoremstyle{theorem}

\theoremstyle{theorem}
\newtheorem{conjecture}[theorem]{Conjecture}

\theoremstyle{corollary}
\newtheorem{corollary}[theorem]{Corollary}

\theoremstyle{remark}
\newtheorem{remark}[theorem]{Remark}

\theoremstyle{remark}

\usepackage{polynom}

\newcommand{\pder}[2]{\frac{\partial#1}{\partial#2}}

\DeclareMathOperator{\Hom}{Hom}

\DeclareMathOperator{\Ext}{Ext}

\usepackage{tikz-cd}
\usetikzlibrary{matrix,arrows,decorations.pathmorphing}
\title[Left-Orderable Surgeries on the knot $6_2$]{Left-Orderable Surgeries on the knot $6_2$ via hyperbolic $\widetilde{PSL}(2,\mathbb{R})$-Representations}
\author{Ollie Thakar}
\address{\parbox{\linewidth}{Department of Mathematics, Princeton University, New Jersey, 08540 \\ \emph{Current Address:} Department of Mathematics, Harvard University, Massachusetts, 02138}}
\email{othakar@math.harvard.edu}

\begin{document}

\maketitle\begin{abstract} We introduce a new method of detecting when the fundamental group of a Dehn surgery on a knot admits a left-ordering, a method which is particularly useful for 2-bridge knots. As an illustration of this method, we show that all Dehn surgeries on the knot $6_2$ with slope in the interval $(-4, 8)\cap\mathbb{Q}$ have left-orderable fundamental groups by exhibiting a family of hyperbolic $\widetilde{PSL}(2,\mathbb{R})$-representations of the knot complement group.
\end{abstract}

\section{Introduction}

The $L$-space conjecture is an ambitious conjecture in 3-manifold topology attempting to unite information about the Heegaard Floer homology of a three-manifold with information about its fundamental group. To state this conjecture, we must first go through some preliminary definitions:

\begin{definition}
A group $G$ is said to be \emph{left-orderable} if it admits a total ordering $<$ such that if $a<b$ then $ca<cb$ for all $a,b,c\in G.$
\end{definition} 

\begin{definition}
A rational homology 3-sphere $Y$ is said to be an \emph{$L$-space} if $$\text{rank} ~\widehat{HF}(Y) = |H_1(Y; \mathbb{Z})|,$$ where $\widehat{HF}(Y)$ is the hat-flavored Heegaard Floer homology of $Y$ as defined in \cite{OS04a}.
\end{definition}

With these preliminaries in place, we may state the conjecture:

\begin{conjecture}[The $L$-Space Conjecture, \cite{BGW}]
Let $Y$ be a rational homology 3-sphere. Then, the following are equivalent:
\begin{enumerate}
	\item $Y$ is an $L$-space.
	\item $\pi_1(Y)$ is not left-orderable.
	\item $Y$ does not admit a taut foliation.
\end{enumerate}
\end{conjecture}

The equivalence of (1) and (2) was first conjectured by Boyer, Rolfsen, and Wiest, while the equivalence of those conditions with (3) was added by Juhasz. The implication (1)$\implies$(3) was shown by Ozsv\'ath and Szab\'o in \cite{OSgenus}, but all the other equivalences remain open.

This paper will primarily focus on the implication (1)$\iff$(2) for the case that $Y$ is a Dehn surgery on a knot. To understand when surgeries on knots are $L$-spaces, we make the following definition:

\begin{definition}
A knot $K\subset S^3$ is an $L$-space knot if there exists some positive slope $r\in\mathbb{Q}$ such that $S^3_r(K)$ is an $L$-space.
\end{definition}

This definition is particularly useful to us on account of the following theorem stated in work of Hedden and Ozsv\'ath-Szab\'o:

\begin{theorem}[\cite{hedden} Lemma 2.13, \cite{OSrationalsurgeries} Proposition 9.5]
Let $K\subset S^3$ be an $L$-space knot, and let $g(K)$ be the Seifert genus of $K$. Then, the surgery $S^3_r(K)$ is an $L$-space if and only if $r\geq 2g(K) - 1.$
\end{theorem}

We now wish to see examples of this definition. Since lens spaces are $L$-spaces and every positive torus knot admits a positive surgery to an $L$-space (see \cite{Moser}), this tells us that positive torus knots are $L$-space knots. In fact, the implication (1)$\iff$(2) is known for all torus knots, and it follows from this above theorem and the proof of (1)$\iff$(2) for Seifert-fibered spaces in \cite{LS} and \cite{BGW}, since every nonzero surgery on a torus knot is Seifert-fibered.

Among alternating knots, in particular 2-bridge knots, torus knots are the only examples of $L$-space knots by \cite[Theorem 1.5]{OSlens}. Hence, all surgeries on non-torus 2-bridge knots are non-$L$-spaces (since 2-bridge knots are closed under taking mirrors and a negative surgery on a knot is homeomorphic to positive surgery on its mirror.) In particular, if we wish to prove the $L$-space conjecture for 2-bridge knots, we hope to show that all nonzero surgeries on non-torus 2-bridge knots have left-orderable fundamental group. 

A powerful and common technique is to use the following theorem of Boyer, Rolfsen, and Wiest.

\begin{theorem}[\cite{BRW} Theorem 1.1.1, Theorem 3.2]
Let $M$ be a compact, irreducible, orientable 3-manifold (possibly with boundary.) Then, $\pi_1(M)$ is left-orderable if and only if $\pi_1(M)$ admits a nontrivial left-orderable quotient group.
\end{theorem}

This theorem suggests that to show a Dehn surgery $S^3_r(K)$ has left-orderable fundamental group for some $r\in\mathbb{Q}$ and knot $K\subset S^3,$ we should try to find a nontrivial representation $\pi_1(S^3_r(K))\to G$ where $G$ is some left-orderable group. A popular choice for $G$ is the group $\widetilde{PSL}(2, \mathbb{R}),$ the universal cover of the special linear group, which is left-orderable since it acts by orientation-preserving homeomorphisms on the real line. This group is difficult to work with directly, so a friendlier strategy is to first exhibit representations into $SL(2, \mathbb{R})$ and then argue that these representations lift to the universal cover. We will focus on this strategy for the remainder of the paper.

We now summarize what is known about the state of left-orderability of surgeries on non-torus 2-bridge knots. The simplest non-torus 2-bridge knot is the figure-eight knot $4_1,$ for which Zung showed that all surgeries are left-orderable \cite{Zung} using a rather different set of techniques involving pseudo-Anosov flows. Beyond this knot, most of what is known in the literature is for \emph{double-twist knots}, a special class of two-bridge knots.

Recall that a 2-bridge knot may be specified as $K(p, q)$ where $p,q$ are odd integers and $q\in(-p, p),$ with $K(p, q)$ and $K(r, s)$ isotopic if and only if $p = r$ and $q\equiv \pm s^{\pm 1} \pmod p.$ (Here, a continued fraction expansion of $q/p$ corresponds to the geometric decomposition of the two-bridge knot in terms of tangles.) We define the double-twist knot $C(k, 2n)$ or $C(k, -2n)$ where $k$ and $n$ are positive integers to be as follows: $C(k, 2n) = K(-1+2nk, k)$ and $C(k, -2n) = K(1+2nk, k)$ if $k$ is odd, and $C(k, 2n) = K(-1+2nk, 1-(2n-1)k)$ and $C(k, -2n) = K(1+2nk, -1-(2n-1)k)$ if $k$ is even. In this notation, the figure-eight knot is the knot $C(2,2).$ Geometrically, the double-twist knots correspond to bridge diagrams of knots that contain only 2 twist regions. The following table displays intervals such that if $r$ is a rational number on that interval, then surgery on the corresponding knot with slope $r$ has a left-ordering:

\begin{tabular}{ c | c | c | c | c | c }
Type & $C(2m+1, 2n)$ & $C(2m+1, -2n)$ & $C(2m, 2n)$ & $C(2m, -2n)$ & reference \\
\hline
elliptic & $(-\infty, 2n-1)$ & $(-\infty, 2n-1)$ & & $(-\infty, 1)$ & \cite{Tran} \\
\hline
hyperbolic &  & $(-4n, 4m)$ & $(-4n, 4m)$ & $[0, \max\{4n, 4m\})$ & \cite{KTT}, \cite{Hakamata} \\
\hline 
elliptic &  & $(-1, \infty)$ if $n\geq2$ & & & \cite{Gao2} \\
\end{tabular}

All of these results were found by the strategy we outlined above: exhibiting representations of the fundamental groups of these surgeries into the group $\widetilde{PSL}(2, \mathbb{R}).$ The ``type'' column on the left refers to whether these representations were hyperbolic or elliptic ones. Yet more is known for the $C(2m, 2n)$ case if we further assume the knot is hyperbolic, as found in \cite{Tranhyp}.

These results rely heavily on a presentation of the fundamental group of the complement of double-twist knots which is particularly simple, such that representations of this group into $SL(2, \mathbb{R})$ can be parameterized by solutions to certain iterated Chebyshev polynomials. These polynomials are very convenient but only work in this special case. Here, I demonstrate a more general method of finding such representations for all 2-bridge knots, and show that my method is effective by expanding the table above. Specifically, I will use my method to prove the following theorem about the knot $6_2,$ which is the double-twist knot $C(3, 4).$

\begin{theorem}\label{62}
If $r\in(-4, 8)\cap\mathbb{Q},$ then $r$-framed Dehn surgery on the knot $6_2$ has left-orderable fundamental group.
\end{theorem}

This theorem improves on the currently known surgery slope for that knot, which is $(-\infty, 3)$ by the above table. That result uses elliptic representations, while my result uses hyperbolic ones. The most powerful theorem in the method that I will lay out actually provides hope for finding hyperbolic $\widetilde{PSL}(2, \mathbb{R})$-representations for surgeries on all knots:

\begin{theorem}\label{hyperbolic}
Let $K\subset S^3$ be a knot. Consider $\pi_1(S^3-K)$ as a topological group with the discrete topology. Suppose that there exists a continuous family of non-abelian representations $\rho_t:\pi_1(S^3 - K)\times [0, 1]\to SL(2, \mathbb{R})$ such that for each fixed $t,$ $\rho_t|_{\langle \mu, \lambda\rangle}$ is hyperbolic. Suppose further that there exists a function $f:[0, 1]\to\mathbb{R}$ such that if $f(t) = p/q\in\mathbb{Q}$ then $\rho_t(\mu^p\lambda^q) = \text{Id}.$ If $f([0, 1])$ contains $1/n$ for some $n\in\mathbb{Z},$ then for each $p/q\in\mathbb{Q}\cap f([0, 1])$ there exists a non-abelian representation $\pi_1(S^3_{p/q}(K))\to \widetilde{PSL}(2, \mathbb{R})$.
\end{theorem}

The next section of this paper will discuss a way to parameterize $SL(2, \mathbb{R})$-representations for surgeries on 2-bridge knots, while Section 3 will show how to lift these representations to the universal cover by proving Theorem \ref{hyperbolic}, a proof which uses group cohomology as its primary tool. Section 4 will prove Theorem \ref{62}, and Section 5 will describe an infinite family of 2-bridge knots admitting a range of left-orderable surgeries that can be derived from the range I show for $6_2.$

\subsection{Acknowledgements}

This paper is based on a section of my undergraduate senior thesis at Princeton University. I would like to thank my adviser, Zolt\'{a}n Szab\'{o}, for facilitating and creating this whole project, and also my second reader, Peter Ozsv\'{a}th.

\section{Riley Polynomials}

Here, we parameterize representations of surgeries on 2-bridge knots. The idea is that if $K= K(p,q)$ is a 2-bridge knot, we can explicitly compute a family of representations of $\pi_1(S^3 - K)$ using a simple presentation of the fundamental group, and these representations will descend to representations of the surgery group. See \cite[Proposition 1]{Rileypar} for details about 2-bridge knots and as a reference for the following presentation of the fundamental group of 2-bridge knots.

For $i\in\mathbb{N}$ with $i\leq p,$ let $e_i = (-1)^{\left\lfloor\frac{iq}{p}\right\rfloor}$ and let $\sigma = \sum_{i=1}^{p-1} e_i.$  Then, the fundamental group $\pi_1(S^3 - K)$ can be presented as a group $$\pi_1(S^3 - K) = \langle x, y | wx=yw \rangle$$ where $$w = x^{e_1}y^{e_2}\dots x^{e_{p-2}}y^{e_{p-1}}.$$ The meridian is given by $x$ and the Seifert longitude is given by $x^{-2\sigma}w_*w$ where $$w_* = y^{e_1}x^{e_2}\dots y^{e_{p-2}}x^{e_{p-1}}.$$ (Riley in \cite{Rileypar} writes this as $w_*wx^{-2\sigma},$ but also specifies that this element commutes with $x.$)

A theorem of Riley (\cite[Theorem 1]{Rileynon}) shows that every non-abelian representation $\rho$ of the two-bridge knot group $\pi_1(S^3 - K)$ into $SL(2, \mathbb{C})$ can be conjugated into one in a standard form where $\rho(x) = \begin{pmatrix} t & 1 \\ 0 & t^{-1} \\ \end{pmatrix}$ and $\rho(y) = \begin{pmatrix} t & 0 \\ -u & t^{-1} \\ \end{pmatrix}$ for some $t, u\in\mathbb{C}.$ Moreover, he also shows that there exists a Laurent polynomial $\Phi(t,u)\in\mathbb{Z}[t, t^{-1}, u]$ dependent on $K(p,q)$ such that the assignment $\rho$ defined on generators is actually a representation if and only if $\Phi(t, u)=0.$ Hence, to find a continuous family of representations into $SL(2, \mathbb{R})$, it suffices to find a continuous family of solutions $(t, u)$ to the Riley polynomial such that the matrices $\begin{pmatrix} t & 1 \\ 0 & t^{-1} \\ \end{pmatrix}$ and $\begin{pmatrix} t & 0 \\ -u & t^{-1} \\ \end{pmatrix}$ can be simultaneously conjugated into real matrices.

All of the results on double-twist knots I listed above find these continuous families of solutions to the Riley polynomial using special simplifications afforded by presentations of double-twist knots. Double-twist knots have simpler group presentations than arbitrary knots, and in fact we can express the Riley polynomial in terms of Chebyshev polynomials, which have many recursion relations and properties we can take advantage of to find solutions easily (see \cite{Tran} for instance.) Here, I show how to derive a more general set of two equations, each solution of which gives a representation of the group $\pi_1(S^3_n(K))$ for some surgery coefficient $n$ into $SL(2, \mathbb{R}).$ One of these equations in the set is the original Riley polynomial from before, and the other is new.

First, we must introduce a lot of notation. Let $t, u$ be complex numbers and suppose $t\neq0.$ Then, we may define the following matrices: $$C = \begin{pmatrix} t & 1 \\ 0 & t^{-1} \\ \end{pmatrix}, D = \begin{pmatrix} t & 0 \\ -u & t^{-1} \\ \end{pmatrix}, X = \begin{pmatrix} t - t^{-1} & 1 \\ -u & t^{-1} - t \\ \end{pmatrix}$$

Furthermore, we follow Riley in \cite{Rileynon} and for $u\neq0$ define $$V = \begin{pmatrix} \alpha & 0 \\ 0 & \alpha^{-1} \\ \end{pmatrix}$$ where $\alpha\in\mathbb{C}$ is a square root of $-u.$ 

We define the matrices: $$W = C^{e_1}D^{e_2}\dots C^{e_{p-2}}D^{e_{p-1}}$$ and $$W_* = D^{e_1}C^{e_2}\dots D^{e_{p-2}}C^{e_{p-1}}.$$

Denote the coordinates of $W$ by $W = \begin{pmatrix} a & b \\ c & d \\ \end{pmatrix}.$

We define a function $\rho:\{x, y\}\mapsto SL(2, \mathbb{C})$ as $\rho(x)=C, \rho(y) = D.$  Clearly, we can extend $\rho$ to a well-defined map $\rho:\pi_1(S^3 - K)\to SL(2, \mathbb{C})$ if $WC=DW$. The miracle of Riley in \cite{Rileynon} is that the condition $WC = DW$ is equivalent to a single (Laurent) polynomial in $t$ and $u.$ We will see this polynomial in the upcoming derivations:

Furthermore, $\rho$ extends to a map $\pi_1(S^3_{n}(K))\to SL(2, \mathbb{C})$ if $WC=DW$ and $C^{n-2\sigma}W_*W = \text{Id}.$ Let $N$ be defined as: $$N = n - 2\sigma.$$ The end goal will be for us to reduce from $SL(2, \mathbb{C})$ to $SL(2, \mathbb{R}).$

\begin{lemma}[Riley,  \cite{Rileynon} p. 194]
For all $i$ we have $e_i = e_{p-i}$.
\end{lemma}

\begin{proof}
The proof follows from a direct computation.
\end{proof}

\begin{lemma}[Riley, \cite{Rileynon} p.195]
We have $c = -ub.$
\end{lemma}

\begin{proof}
When $u=0,$ then $W$ will be a product of upper-triangular matrices, so $W$ will be an upper-triangular matrix; hence $c=0=-ub.$

Otherwise, the matrix $V$ is well-defined and invertible. Observe that $VCV^{-1} = D^t$ and $VDV^{-1} = C^t.$ (For all values of $t$ and $u$ we know Hence, by the definition of $W$ and the above lemma, $VWV^{-1} = W^t$. Conjugating by the matrix $V$ multiplies the upper-right term of a matrix by $-u$, so comparing upper-right terms of this equation we get the desired result.
\end{proof}

\begin{lemma}[Riley, \cite{Rileynon} p. 195]
$WC = DW$ if and only if $a = (t - t^{-1}) b.$
\end{lemma}

\begin{proof}
The 4 coordinate equations in the matrix equation $WC = DW$ are $a = (t - t^{-1}) b$, trivial, $c = -ub$, and a fourth equation which follows from the earlier ones.
\end{proof}

\begin{lemma}
For $t\neq\pm1$, and $(t-t^{-1})^2\neq u,$ we have that $C^{N}W_*W = \text{Id}$ and $WC = DW$ if and only if $a = (t - t^{-1}) b$ and $t^N((t-t^{-1})^2-u)b^2=1.$
\end{lemma}

\begin{proof}
When $(t-t^{-1})^2\neq u,$ the matrix $X$ is invertible. Note that $XCX^{-1} = D$ and $XDX^{-1} = C,$ hence $XWX^{-1} = W_*.$ Therefore, solving $C^{N}W_*W = \text{Id}$ is equivalent to solving $C^{N}XW = W^{-1}X.$ This is 4 equations; we may use the equations $a = (t - t^{-1}) b$ and $c = -ub$ to replace instances of $a$ and $c$ with those of $b.$

Since $W$ obviously has determinant 1, we also get that $$(t-t^{-1})bd+ub^2=\det(W)=1.$$ 

Now, the 4 equations reduce to 1 trivial one, 2 copies of the equation $t^N((t-t^{-1})^2-u)b^2=1,$ and a fourth equation which is equivalent to $t^N((t-t^{-1})^2-u)b^2=1$ if $t\neq t^{-1},$ which is true whenever $t\neq1.$
\end{proof}

\begin{lemma}
For $t^N=1$ there are no solutions to $C^{N}W_*W = \text{Id}$ and $WC = DW$.
\end{lemma}

\begin{proof}
Suppose that there is some solution to these equations with $t=\pm 1.$ From $WC=DW$ we get that $a=0.$ From $\det(W) = 1$ we get the equation $ub^2=1.$ This tells us in particular that $u$ and $b$ must not vanish.

From the equation $C^{N}W_*W = \text{Id}$, we get, as before, that $C^{N}XW = W^{-1}X.$ The upper-left term of this equation reads $ub=-ub,$ thus there are no solutions.

\end{proof}

\begin{remark}
For $t = -1,$ with $N$ odd, there can be solutions to the equations $C^{N}W_*W = \text{Id}$ and $WC = DW$, however now they exist if and only if the equations $a = 0$ and $-Nub = 2d$ have a mutual solution $u.$
\end{remark}

\begin{lemma}
For $(t-t^{-1})^2 = u,$ there are no solutions to $C^{N}W_*W = \text{Id}$ and $WC = DW$.
\end{lemma}

\begin{proof}
Suppose there exists a solution. $X$ is not invertible, however it is still true that $XC=DX$ and $XD=CX,$ hence $W_*X = XW.$ Thus,  $C^NW_*W=\text{Id}$ still implies that $C^NXW = W^{-1}X,$ forcing $t^N((t-t^{-1})^2-u)b^2=1.$ This is clearly impossible, hence there cannot be a solution.
\end{proof}

For each $p,q,$ let $A_{p,q}(t,u)\in\mathbb{Z}[t, t^{-1}, u]$ denote the expression $a$ used above, $B_{p,q}(t,u)\in\mathbb{Z}[t, t^{-1}, u]$ denote the expression $b$ used above, and $D_{p,q}(t,u)\in\mathbb{Z}[t, t^{-1}, u]$ denote the expression $d$ used above.

\begin{definition}
The \emph{Riley polynomial} of a 2-bridge knot $K(p, q)$ is $A_{p,q}(t,u) - (t-t^{-1})B_{p,q}(t,u).$
\end{definition}

Note that the vanishing of the Riley polynomial at $(t, u)$ gives a concrete representation of the knot complement group in $SL(2, \mathbb{C}).$

We now prove the following theorem:

\begin{theorem}\label{2bridge}
Let $K=K(p,q)$ be a 2-bridge knot. For any $n\in\mathbb{Z},$ the fundamental group $\pi_1(S^3_{n}(K))$ admits a non-trivial $SL(2, \mathbb{R})$ representation if the system of equations: $$A_{p,q}(t, u) = (t - t^{-1}) B_{p,q}(t, u), $$$$t^{n-2\sigma}((t-t^{-1})^2-u)\left(B_{p,q}(t, u)\right)^2=1$$ have a solution $(t, u)\in\mathbb{R}^2$ with $t\notin\{-1, 0, 1\}$ or a solution with $t=e^{i\theta}, t\neq\pm1$ and $u\in(-\infty, -4\sin^2(\theta))\cup(0,\infty),$ or if $n$ is odd and the system of equations: $$A_{p,q}(-1, u) = 0,$$ $$-(n-2\sigma)uB_{p,q}(-1, u) = 2D_{p,q}(-1, u)$$ has a solution $u\in\mathbb{R}.$

Moreover, if $n=\pm1,$ then the existence of a solution to one of the above systems of equations is a necessary and sufficient condition for $\pi_1(S^3_{n}(K))$ to admit a non-trivial $SL(2, \mathbb{R})$ representation.
\end{theorem}

\begin{proof}
By \cite[p. 786-787]{Khoi}, a representation of $\langle x,y | wx=yw\rangle$ of the form $x\mapsto C, y\mapsto D$ is conjugate to a $SL(2, \mathbb{R})$ representation (Khoi discusses $SU(1,1)$ instead of $SL(2, \mathbb{R})$ but these groups are isomorphic) if and only if one of the 2 below conditions holds:

\begin{itemize}
	\item $t,u\in\mathbb{R}, t\neq0.$
	\item $t=e^{i\theta}, t\neq1, u\in(-\infty, -4\sin^2(\theta))\cup(0,\infty).$
\end{itemize}

Since every representation of $\pi_1(S^3_n(K))$ we discuss comes from quotient-ing a representation of $\langle x,y | wx=yw\rangle$ of the above form, the result in combination with our above lemmas provides us with the first part of the theorem.

We now prove the second part of the theorem. Sufficiency follows from the first part, so we focus on necessity. Suppose $\pi_1(S^3_{\pm1}(K))$ admits a non-trivial $SL(2, \mathbb{R})$ representation. Then, it must be non-abelian since $$\pi_1(S^3_{\pm1}(K))/[\pi_1(S^3_{\pm1}(K)), \pi_1(S^3_{\pm1}(K))] = H_1(S^3_{\pm1}(K)) = 0.$$ Furthermore, it induces an $SL(2, \mathbb{C})$ representation of $\pi_1(S^3 - K)$ via the maps: $$\pi_1(S^3 - K)\to\pi_1(S^3 - K)/\langle\mu^{\pm1}\lambda=1\rangle\cong\pi_1(S^3_{\pm1}(K))\to SL(2, \mathbb{R})\hookrightarrow SL(2, \mathbb{C})$$ which is evidently non-abelian. 

\cite[Theorem 1]{Rileynon} proves that every non-abelian $SL(2, \mathbb{C})$-representation of a 2-bridge knot group can be conjugated (by a matrix in $GL(2, \mathbb{C})$) to a scalar multiple of a representation sending $x$ to $C(t)$ and $y$ to $D(t,u)$ for some $t,u\in\mathbb{C}$ with $t\neq0.$ Hence, there exists $t,u\in\mathbb{C}$, $\alpha\in\{\pm1\}$ and a map $\rho:\pi_1(S^3 - K)\to SL(2, \mathbb{C})$ with $\rho(x) = \alpha C(t)$ and $\rho(y) = \alpha D(t,u)$. By further conjugating by the matrix $\begin{pmatrix} i & 0 \\ 0 & -i \\ \end{pmatrix}$ if necessary, we may assume $\alpha=1.$ Since this representation is conjugate to an $SL(2, \mathbb{R})$-representation, the aforementioned result from \cite{Khoi} tells us our solution must be of the form: \begin{itemize}
	\item $t,u\in\mathbb{R}, t\neq0.$
	\item $t=e^{i\theta}, t\neq1, u\in(-\infty, -4\sin^2(\theta))\cup(0,\infty).$
\end{itemize}

Furthermore, $\rho(\mu^{\pm1}\lambda)=\rho(1)$ since the map descends to a map on the quotient $\pi_1(S^3 - K)/\langle\mu^{\pm1}\lambda=1\rangle,$ so the equations $WC=DW$ and $C^{\pm1-2\sigma}W_*W=\text{Id}$ must be true. By our above lemmas, it is therefore not possible for $t$ to equal $1.$ If $t=-1,$ these matrix equations hold if and only if $$A_{p,q}(-1, u) = 0,$$ $$-(n-2\sigma)uB_{p,q}(-1, u) = 2D_{p,q}(-1, u).$$ For $t\neq-1,$ these equations hold if and only if $$A_{p,q}(t, u) = (t - t^{-1}) B_{p,q}(t, u), $$$$t^{n-2\sigma}((t-t^{-1})^2-u)\left(B_{p,q}(t, u)\right)^2=1,$$ and that suffices for the proof.
\end{proof}

\begin{corollary}
Let $K=K(p,q)$ be a non-torus 2-bridge knot.The surgery $S^3_{\pm1}(K)$ has a left orderable fundamental group if the system of equations: $$A_{p,q}(t, u) = (t - t^{-1}) B_{p,q}(t, u), $$$$t^{\pm1-2\sigma}((t-t^{-1})^2-u)\left(B_{p,q}(t, u)\right)^2=1$$ have a solution $(t, u)\in\mathbb{R}^2$ with $t\neq0$ or a solution with $t=e^{i\theta}, t\neq1$ and $u\in(-\infty, -4\sin^2(\theta))\cup(0,\infty),$ or if the system of equations: $$A_{p,q}(-1, u) = 0,$$ $$-(\pm1-2\sigma)uB_{p,q}(-1, u) = 2D_{p,q}(-1, u)$$ has a solution $u\in\mathbb{R}.$
\end{corollary}

\begin{proof}
By the above theorem, if one of those systems of equations has a solution, the surgery $S^3_{\pm1}(K)$ admits a non-trivial $SL(2, \mathbb{R})$ representation. Since $S^3_{\pm1}(K)$ is an orientable integral homology 3-sphere and it is irreducible by \cite{HatcherThurston}, Theorem \ref{zhs}, which I will prove in the following section, tells us that it must have a left orderable fundamental group, as desired.
\end{proof}

\begin{remark}
It is rather remarkable that the 8 equations contained in $C^{N}W_*W = \text{Id}$ and $WC = DW$ reduce down to 2 equations. This is perhaps unsurprising in the wake of the fabulous result of Kronheimer and Mrowka in \cite{KM} using gauge theory that establishes that $\pi_1(S^3_1(K))$ admits a non-trivial $SO(3)$-representation for all knots $K\subset S^3.$
\end{remark}

\subsection{Rational Surgeries}

The above method also works to produce rational surgeries. Let $r\in\mathbb{Q}$ be a rational surgery slope, and pick coprime $\alpha,\beta\in\mathbb{Z}$ such that $r - 2\sigma = \alpha/\beta.$ For $t\neq\pm1,$ consider the matrix $$C_{\alpha/\beta} := \begin{pmatrix} t^{\alpha/\beta} & \frac{t^{\alpha/\beta} - t^{-\alpha/\beta}}{t - t^{-1}} \\ 0 & t^{-\alpha/\beta}\end{pmatrix}.$$ It is easy to show that $C_{\alpha/\beta}^\beta = C^\alpha.$

Suppose that the first equation $A_{p,q}(t, u) = (t - t^{-1}) B_{p,q}(t, u)$ is satisfied. Then, I claim that $C_{\alpha/\beta}$ commutes with $W_*W.$ We know that $C$ commutes with $W_*W$ since the first equation being satisfied tells us that $x, y \mapsto C, D$ defines a representation of $\pi_1(S^3 - K),$ and we know that the longitude $x^{-2\sigma}w_*w\in\pi_1(S^3 - K)$ commutes with $x.$ But, $C^\alpha$ and $C_{\alpha/\beta}$ are simultaneously diagonalizable since $C_{\alpha/\beta}^\beta = C^\alpha.$ The eigenvalues of $C_{\alpha/\beta}$ are $t^{\alpha/\beta}$ and $t^{-\alpha/\beta}$ and the eigenvalues of $C$ are $t^\alpha$ and $t^{-\alpha}.$ Therefore, a polynomial $p\in\mathbb{R}[x]$ such that $p(t^{\alpha}) = t^{\alpha/\beta}$ and $p(t^{-\alpha}) = t^{-\alpha/\beta}$ must satisfy $p(C^\alpha) = C_{\alpha/\beta}.$ Thus, $C_{\alpha/\beta}$ is expressible as a polynomial of the matrix $C^\alpha,$ so since $W_*W$ commutes with $C^\alpha,$ it must also commute with $C_{\alpha/\beta}.$

Since $C_{\alpha/\beta}$ commutes with $W_*W,$ then if the equation $C_{\alpha/\beta}W_*W = \text{Id}$ holds by exponentiating we get that $C^\alpha W_*W^\beta = \text{Id}$ holds. This tells us the representation of $\pi_1(S^3 - K)$ descends to the group of the $r$-surgery $\pi_1(S^3_r(K)).$ Proceeding with analogous computations to those done above, we arrive at the following theorem:

\begin{theorem}
Let $K=K(p,q)$ be a 2-bridge knot. For any $r\in\mathbb{Q},$ the fundamental group $\pi_1(S^3_{r}(K))$ admits a non-trivial $SL_2(\mathbb{R})$ representation if the system of equations: $$A_{p,q}(t, u) = (t - t^{-1}) B_{p,q}(t, u), $$$$t^{r-2\sigma}((t-t^{-1})^2-u)\left(B_{p,q}(t, u)\right)^2=1$$ have a solution $(t, u)\in\mathbb{R}^2$ with $t\notin\{-1, 0, 1\}$ or a solution with $t=e^{i\theta}, t\neq\pm1$ and $u\in(-\infty, -4\sin^2(\theta))\cup(0,\infty),$ or if $r,$ expressed in lowest terms, has odd numerator and denominator and the system of equations: $$A_{p,q}(-1, u) = 0,$$ $$-(r-2\sigma)uB_{p,q}(-1, u) = 2D_{p,q}(-1, u)$$ has a solution $u\in\mathbb{R}.$
\end{theorem}

\begin{corollary}
Let $K=K(p,q)$ be a 2-bridge knot. For any $n\in\mathbb{Z} - \{0\},$ the surgery $S^3_{1/n}(K)$ has a left orderable fundamental group if the system of equations: $$A_{p,q}(t, u) = (t - t^{-1}) B_{p,q}(t, u), $$$$t^{\frac1n-2\sigma}((t-t^{-1})^2-u)\left(B_{p,q}(t, u)\right)^2=1$$ have a solution $(t, u)\in\mathbb{R}^2$ with $t\neq0$ or a solution with $t=e^{i\theta}, t\neq1$ and $u\in(-\infty, -4\sin^2(\theta))\cup(0,\infty),$ or if the system of equations: $$A_{p,q}(-1, u) = 0,$$ $$-(\frac1n-2\sigma)uB_{p,q}(-1, u) = 2D_{p,q}(-1, u)$$ has a solution $u\in\mathbb{R}.$
\end{corollary}

\section{Lifting Continuous Families of Representations}

In this section, we focus on how to lift representations $G\to SL(2, \mathbb{R})$ to the universal cover $\widetilde{PSL}(2, \mathbb{R}).$ The main technique is group cohomology. Recall that the second group cohomology $H^2(G; \mathbb{Z})$ is in 1-to-1 correspondence with equivalence classes of central extensions of $G$ by $\mathbb{Z}$, with the trivial element corresponding to a split extension (see \cite{groupcohomology} Section IV.3.) In fact, to each representation $\rho:G\to SL(2, \mathbb{R})$ we may associate a cohomology class $e(\rho)\in H^2(G; \mathbb{Z})$ which is trivial if and only if $\rho$ lifts to a map $\widetilde{\rho}:G\to \widetilde{PSL}(2, \mathbb{R})$. Following \cite[Section 3.5]{CD}, we call $e(\rho)$ the \emph{Euler class} in analogy to the Euler class of an oriented vector bundle.

To construct the Euler class of $\rho$ explicitly, we follow \cite{CD}, which essentially walks us through the correspondence between second group cohomology and equivalence classes of central extensions. Let $\sigma:G\to\widetilde{PSL}(2, \mathbb{R})$ be any section of $\rho.$ Then, for $g,h\in G,$ we know $$\sigma(g)\sigma(h)\sigma(gh)^{-1} \in \pi^{-1}(\text{Id}) = \mathcal{Z}(\widetilde{PSL}(2, \mathbb{R}))\cong\mathbb{Z}.$$ Consider the function $\varphi_\sigma:G\times G\to\mathbb{Z}$ sending $(g,h)$ to the value $\sigma(g)\sigma(h)\sigma(gh)^{-1}\in \mathbb{Z}.$ Intuitively, $\varphi_\sigma$ measures the failure of $\sigma$ to be a homomorphism. We can verify that $\varphi_\sigma$ is indeed a cocycle, and modulo a coboundary it is independent of the choice of section. Hence, we may define $e(\rho) := [\varphi_\sigma]\in H^2(G; \mathbb{Z}).$ It is clear that $e(\rho)$ vanishes if and only if there exists a lift of $\rho,$ since a lift of $\rho$ is a section that is a homomorphism.

Note also that for any lift $\widetilde{\rho},$ the function $\psi\cdot\widetilde{\rho}$ is also a lift where $\psi:G\to\mathbb{Z}$ is a homomorphism.

Luckily for us, the groups that we will be encountering in our study of three-manifolds have very simple group cohomology. The manifolds we will be dealing with are Dehn surgeries on knots, and their fundamental groups are closely related to knot groups. Indeed, the fundamental group $\pi_1(S^3_{p/q}(K))$ is isomorphic to $\pi_1(S^3 - K)/\langle\langle\mu^p\lambda^q\rangle\rangle,$ where $\langle\langle\mu^p\lambda^q\rangle\rangle$ refers to the subgroup normally generated by $\mu^p\lambda^q,$ where $\mu$ is the meridian of the knot and $\lambda$ is the Seifert longitude. We can compute their group cohomologies as follows:

\begin{theorem}\label{coh}
Let $p/q\in\mathbb{Q}$ be a nonzero surgery slope in lowest common terms, and suppose $K\subset S^3$ is a knot. Let $M = S^3_{p/q}(K)$. The second group cohomology $H^2(\pi_1(M); \mathbb{Z})$ is isomorphic to $\mathbb{Z}/p\mathbb{Z}.$
\end{theorem}

\begin{proof}
By Hopf's Theorem (see \cite[Theorem II.5.2]{groupcohomology},) we have that $$H_{2, group}(\pi_1(M), \mathbb{Z}) \cong H_{2}(M; \mathbb{Z})/h(\pi_2(M)),$$ where $h:\pi_2(M)\to H_2(M; \mathbb{Z})$ is the Hurewicz homomorphism. We know also by \cite{groupcohomology} (or the Hurewicz theorem) that $$H_{1, group}(\pi_1(M); \mathbb{Z}) = H_{1, singular}(M; \mathbb{Z}) = \mathbb{Z}/p\mathbb{Z}.$$ 

By the universal coefficient theorem and Poincar\'{e} duality, $H_2(M; \mathbb{Z})$ vanishes. Hence, by the universal coefficient theorem, $H^2_{group}(\pi_1(M), \mathbb{Z})$ is given by $\Ext(H_1(M; \mathbb{Z}), \mathbb{Z}) \cong \mathbb{Z}/p\mathbb{Z}.$ 
\end{proof}

\begin{theorem}\label{zhs}
Suppose $M$ is an integral homology 3-sphere (such as $1/n$ surgery on a knot), and there exists a nontrivial homomorphism $\pi_1(M)\to PSL(2, \mathbb{R})$ or $\pi_1(M)\to SL(2, \mathbb{R})$. Then this homomorphism lifts to $\widetilde{PSL}(2, \mathbb{R}).$
\end{theorem}

\begin{proof}
By the computation of Theorem \ref{coh}, we have that $H^2_{group}(\pi_1(M), \mathbb{Z})$ vanishes. Hence, the Euler class of any map $\pi_1(M)\to PSL(2, \mathbb{R})$ or $\pi_1(M)\to SL(2, \mathbb{R})$ vanishes, so they lift to $\widetilde{PSL}(2, \mathbb{R}).$
\end{proof}

Now, we are ready to prove Theorem \ref{hyperbolic}:

\begin{proof}[Proof of Theorem \ref{hyperbolic}]
Let $t_0\in[0,1]$ such that $f(t_0)=1/n$ for some $n\in\mathbb{Z}.$ Note that $\rho_{t_0}:\pi_1(S^3 - K)\to SL(2, \mathbb{R})$ must lift to a representation $\widetilde{\rho}_{t_0}:\pi_1(S^3 - K)\to \widetilde{PSL}(2, \mathbb{R})$ since $H^2_{group}(\pi_1(S^3 - K); \mathbb{Z})$ is trivial as in the previous chapter. Furthermore, we know by hypothesis that $\rho_{t_0}$ descends to a representation of $\pi_1(S^3_{1/n}(K)) \cong \pi_1(S^3 - K)/\langle\langle\mu\lambda^n\rangle\rangle.$ This representation also lifts to $\widetilde{PSL}(2, \mathbb{R})$ since $H^2(\pi_1(S^3_{1/n}(K)); \mathbb{Z})$ is trivial from our computation in the previous chapter. This tells us that we can choose $\widetilde{\rho}_{t_0}(\mu\lambda^n)$ to be $1_{\widetilde{PSL}(2, \mathbb{R})}.$ 

Let $\pi:\widetilde{PSL}(2, \mathbb{R})\to SL(2, \mathbb{R})$ be the projection homomorphism and let $s\in\widetilde{PSL}(2, \mathbb{R})$ be a generator of the center $\mathcal{Z}(\widetilde{PSL}(2, \mathbb{R})) = \pi^{-1}(\text{Id}).$ Now, for each $g\in\pi_1(S^3 - K)$ there exists a unique path $\widetilde{\rho}_t(g):[0, 1]\to\widetilde{PSL}(2, \mathbb{R})$ such that $\widetilde{\rho}_{t_0}(g)$ is specified as above, by the homotopy lifting property. I wish to show that for each fixed $t$, that $$\widetilde{\rho}_t:\pi_1(S^3 - K)\to \widetilde{PSL}(2, \mathbb{R})$$ is a homomorphism of groups. If $g,h\in\pi_1(S^3 - K),$ then the function $f:[0, 1]\to \widetilde{PSL}(2, \mathbb{R})$ given as: $$t\mapsto \widetilde{\rho}_t(g)\widetilde{\rho}_t(h)\widetilde{\rho}_t(gh)^{-1}$$ is clearly continuous. Since $\pi\circ f(t) = \text{Id},$ the image of $f$ is the center $\mathcal{Z}(\widetilde{PSL}(2, \mathbb{R})),$ which is discrete. Hence, $f(t)=f(0)=\text{Id}$ for all $t,$ proving my claim. Hence, we have a continuous family of representations $\rho_t:\pi_1(S^3 - K)\times[0, 1]\to \widetilde{PSL}(2, \mathbb{R})$ such that for each slope $p/q\in f([0, 1])\cap\mathbb{Q}$ there exists $t$ such that $\widetilde{\rho}_{t}(\mu^p\lambda^q) \in \mathcal{Z}(\widetilde{PSL}(2, \mathbb{R})).$

Recall that if $\widetilde{\rho}:\pi_1(S^3 - K)\to \widetilde{PSL}(2, \mathbb{R})$ is a lift of a representation $\rho:\pi_1(S^3 - K)\to SL(2, \mathbb{R})$, then so is $\varphi\cdot\widetilde{\rho}$ where $\varphi\in\Hom(\pi_1(S^3 - K), \mathcal{Z}(\widetilde{PSL}(2, \mathbb{R})))$ (see \cite[Section 3.4]{CD} for instance.) Note further that $\pi_1(S^3 - K)$ abelianizes to $H_1(S^3 - K) \cong\mathbb{Z},$ which is generated by $\mu.$ Hence, all lifts are the same on $\lambda,$ as $\lambda$ is trivial in the abelianization of $\pi_1(S^3 - K).$

By continuity, $\widetilde{\rho}_{t}(\mu\lambda^n)$ must lie in the connected component of hyperbolic elements of $\widetilde{PSL}(2, \mathbb{R})$ containing the identity $1_{\widetilde{PSL}(2, \mathbb{R})}.$ Call this component $U.$ By Khoi in \cite[p. 764]{Khoi}, each element in this connected component lies in a subgroup $G_t$ of $\widetilde{PSL}(2, \mathbb{R})$ isomorphic as a topological group to $\mathbb{R},$ which is contained in $U.$ Furthermore, Khoi also shows that $\widetilde{\rho}_{t}(\mu)$, as it commutes with $\widetilde{\rho}_{t}(\mu\lambda^n)$, can be assumed to also lie in $s^k$ times an element of this subgroup. There exists a function $\varphi\in\Hom(\pi_1(S^3 - K), \mathcal{Z}(\widetilde{PSL}(2, \mathbb{R})))$ such that $\varphi(\mu) = s^{-k}.$ Hence, if we multiply by the appropriate $\varphi\in\Hom(\pi_1(S^3 - K), \mathbb{Z}),$ we may assume that $\widetilde{\rho}_{t}(\mu)$ is also in $G_t.$

Similarly, $\widetilde{\rho}_{t}(\lambda)$ commutes with $\widetilde{\rho}_{t}(\mu\lambda^n)$ so it can be assumed to lie in $s^k$ times an element of $G_t$ for some $k.$ Since $\widetilde{\rho}_{t}(\lambda^n)$ is in $G_t,$ we must have $s^{kn}=1$ so $k=0,$ telling us $\widetilde{\rho}_{t}(\lambda)\in G_t.$ If $\rho_{t}(\mu^p\lambda^q) = \text{Id},$ then we must have $$\widetilde{\rho}_{t}(\mu^p\lambda^q)\in \mathcal{Z}(\widetilde{PSL}(2, \mathbb{R}))\cap G_t = \{1_{\widetilde{PSL}(2, \mathbb{R})}\},$$ showing that there exists a non-abelian $\widetilde{PSL}(2, \mathbb{R})$-representation of $\pi_1(S^3_{p/q}(K))$ when $p/q\in f([0, 1])\cap\mathbb{Q}.$ That suffices for the proof.
\end{proof}

\section{The Knot $6_2$}

The knot $6_2$ is the double-twist knot $C(3, 4)$, or the 2-bridge knot $K(11, 3).$ From the survey of known results, we know that the fundamental group of surgery on this knot is left-orderable when the slope is in the range $(-\infty, 3),$ and this is shown using a continuous family of elliptic representations in \cite{Tran}. 

Here, I demonstrate the power of the method I have outlined in the previous two sections by proving that any slope in the range $(-4, 8)$ on $K(11, 3)$ has left-orderable fundamental group by exhibiting a continuous family of hyperbolic representations. Note that $0$-surgery on $K(11,3)$ has $\mathbb{Z}$-first homology so it is left-orderable by the Boyer-Rolfsen-Wiest Theorem. Hence, we focus on non-zero surgeries in this interval.

We introduce a bit more notation; we define the following polynomials which appeared in the statement of Theorem \ref{2bridge}: $$P(t,u) = A_{p,q}(t, u) - (t - t^{-1}) B_{p,q}(t, u), $$$$Q(t, u) = t^{n-2\sigma}((t-t^{-1})^2-u)\left(B_{p,q}(t, u)\right)^2-1.$$

\begin{theorem}
For any surgery in the interval $(-4, 0]$ on $K$, there exists a non-abelian $SL_2(\mathbb{R})$ representation of the fundamental group.
\end{theorem}

\begin{proof}
Step 1. For $t \geq \frac12\left(\sqrt{\frac{\sqrt{5}-1}2} + \sqrt{\frac{\sqrt{5}+7}2}\right)$ then there exists a continuous function $\psi(t)$ with $-t^{-4}<\psi(t)\leq 0$ for all $t$ such that $P(t, \psi(t))=0.$

Call $t_{\text{min}} = \frac12\left(\sqrt{\frac{\sqrt{5}-1}2} + \sqrt{\frac{\sqrt{5}+7}2}\right).$ For $t \geq t_{\text{min}}$, we compute directly that $P(t, 0)$ is nonpositive and that $P(t, -t^{-4})$ is positive. Indeed, we can find their roots of $P(t, 0)$ explicitly since it is expressible as a polynomial in $u$ and $(t-t^{-1})^2$ of degree $\leq4.$ For $P(t, -t^4)$, it is immediately positive for all $t>1$ (note $t_{\text{min}}>1$.) Hence, we may define $\psi(t)$ to be the smallest solution $u\in (-t^{-4}, 0]$ to $P(t, u)=0.$

To show $\psi(t)$ is continuous, it therefore suffices by the implicit function theorem to show that the partial derivative $\pder{P}{u}$ is non-vanishing in the strip $t > t_{\text{min}},$ $0\geq u\geq -t^{-4}.$ But, we compute this explicitly, letting $y = (t-t^{-1})^2$: $$\pder{P}{u} = (-y^4-5y^3-8y^2-2y+1)+2u(4y^3+15y^2+16y+3)-3u^2(6y^2+15y+8)+4u^3(4y+5)-5u^4.$$ This is evidently negative for $u\leq0$ and $t>t_{\text{min}}.$

Now, the function $n:(t_{\text{min}}, \infty)\to\mathbb{R}$ determined by $$t^{n(t)-4}(t^2+t^{-2}-\psi(t)-2)B_{11, 3}^2(t, \psi(t)) = 1$$ is clearly well-defined and continuous since for $t\in(t_{\text{min}}, \infty)$ the left-hand side does not vanish (indeed $B_{p, q}(t, u)$ never vanishes when $P(t, u)=0$ since this would imply $A_{p, q}(t, u)$ also vanishes which would imply that the determinant of $W$ vanishes, a contradiction.)

Moreover, since $\psi(t) = O(t^{-4})$ as $t\to\infty,$ then $$B_{11, 3}(t, \psi(t)) = t^3 - \psi(t)t^5 + (2 \psi(t)+3\psi^2(t))t^3 - (3\psi^3(t)+4\psi^2(t)+5\psi(t)+1)t = t^3(1 + o(1)).$$

Hence, $(t^2+t^{-2}-\psi(t)-2)B_{11, 3}^2(t, \psi(t)) = t^8(1+o(1))$, which proves that $\lim_{t\to\infty} n(t) = -4.$ Since $\psi(t_{\text{min}}) = 0$ we may directly compute that $n(t_{\text{min}}) = 0,$ so since $n(t)$ is continuous, it assumes every value in the interval $(-4, 0].$
\end{proof}

\begin{theorem}
For any surgery in the interval $(0, 8)$ on $K$, there exists a non-abelian $SL_2(\mathbb{R})$ representation of the fundamental group.
\end{theorem}

\begin{proof}
Step 1. For $u\geq0$ then there exists a continuous function $\varphi(u)$ with $$\frac12\left(\sqrt{u}+\sqrt{u+4}\right)<\varphi(u)<\frac12\left(\sqrt{u+1}+\sqrt{u+5}\right)$$ such that $P(\varphi(u), u)=0.$

By Riley \cite[Lemma 4]{Rileynon}, we know that $P(t, (t-t^{-1})^2) = 1$ for all knots. Hence, setting $t = \frac12\left(\sqrt{u}+\sqrt{u+4}\right),$ we get that $P(t, u) = 1.$ We can also compute directly that $P(t, (t-t^{-1})^2-1)$ is negative whenever $t\geq \frac{1+\sqrt{5}}2,$ hence for $u\geq0,$ setting $t = \frac12\left(\sqrt{u+1}+\sqrt{u+5}\right),$ we get that $P(t, u)<0.$ Thus, we have shown there exists a solution for each fixed $u$ in the desired range. By the implicit function theorem, it suffices to show $\pder Pt\neq0$ in that range. But, we can compute that for all $t,$ we have $$\pder Pt(t, (t-t^{-1})^2-x) = -(8x^3+30x^2+32x+8)(t^3-t^{-5})+(8x^4+46x^3+92x^2+68x+14)(t-t^{-3}),$$ which I claim is negative whenever $t>1$ and $(t-t^{-1})^2\geq x.$ Indeed, when $(t-t^{-1})^2= x$ this expression simplifies to $-4(t^3-t^{-5})+6(t-t^{-3})$ which is negative whenever $t>1$; then, when we increase $t$ and fix $x$ the expression clearly decreases.

Thus, $\varphi(u)$ exists and is smooth. Moreover, for sufficiently large $t$ we also may compute that $P(t, (t-t^{-1})^2-t^{-2})$ is negative. This puts a stronger bound on $\varphi(u).$ We compute that for $x\in[0,1],$ $$B_{11, 3}(t, (t-t^{-1})^2-x) = -(x^3+4x^2+4x)t+O(t^{-1}),$$ so as $u$ approaches infinity, $t=\varphi(u)$ also approaches infinity and since $x<t^{-2}$ for sufficiently large $u, t$ we get $B_{11, 3}(\varphi(u), u) = O(t^{-1}).$ Hence, letting $n:(0, \infty)\to\mathbb{R}$ be determined by $$\varphi^{n(u)-4}(u)(\varphi^2(u)+\varphi^{-2}(u)-u-2)B_{11, 3}^2(t, u) = 1,$$ this is a continuous function as before, and $n(0)=0$ by direct computation. We have computed that $\varphi^2(u)+\varphi^{-2}(u)-u-2 = O(\varphi^{-2}(u))$ as $u\to\infty,$ therefore, with the estimate of $B_{11,3},$ we get $\liminf_{u\to\infty} n(u) \geq 8.$ Since $n(u)$ is continuous, it assumes every value in the interval $[0, 8).$\end{proof}

Note that in the above 2 theorems we have proven actually a bit more. For the negative slopes, we have a continuous family of representations $\rho_t:\pi_1(S^3 - K)\times[t_{\text{min}}, \infty)\to SL(2, \mathbb{R})$ such that for each slope $r\in(-4, 0]\cap\mathbb{Q}$ there exists $t(r)$ such that $\rho_{t(r)}(\mu^p\lambda^q) = 1.$ In fact, we know even more than this. Each $t$ in this representation $\rho_t$ is $>1,$ hence if $x=\mu\in\pi_1(S^3 - K)$ is a meridian then $\rho_t(\mu) = \begin{pmatrix} t & 1 \\ 0 & t^{-1} \\ \end{pmatrix}$ is clearly hyperbolic. Hence, the restriction of each representation to the peripheral subgroup $\langle\mu, \lambda\rangle$ (which is abelian) is hyperbolic.

The analogous result holds for the positive surgeries. Observing that these two continuous families intersect at $(t, u)=(t_{\text{min}}, 0),$ we can reparameterize the representations as follows. We get a continuous family of representations $\rho_\tau:\pi_1(S^3 - K)\times(-1, 1)\to SL(2, \mathbb{R})$ such that for each slope $r\in(-4, 8)\cap\mathbb{Q}$ there exists $\tau(r)$ such that $\rho_{\tau(r)}(\mu^p\lambda^q) = 1,$ and such that the restriction of the representation to $\langle\mu, \lambda\rangle\subset \pi_1(S^3 - K)$ is hyperbolic.

Hence, we are ready to prove Theorem \ref{62}. The below theorem tells us a little bit more than the statement of Theorem \ref{62} in the introduction:

\begin{theorem}
For any slope $r\in(-4, 8)\cap\mathbb{Q},$ the fundamental group $\pi_1(S^3_r(K(11, 3)))$ admits a non-abelian $\widetilde{PSL}(2, \mathbb{R})$ representation, and hence it is left-orderable.
\end{theorem}

\begin{proof}
By the above discussion, we have two continuous families of hyperbolic representations of $\pi_1(S^3 - K(11,3))$ such that for each $p/q\in(-4, 0)\cup(0, 8)$ there exists some representation that is trivial on $\mu^p\lambda^q.$ (One family is for the negative slopes, the other is positive.) Each of these families contains at least one representation that is trivial on $\mu\lambda^n$ for $n=\pm1\in\mathbb{Z},$ hence they meet the conditions of Theorem \ref{hyperbolic}. This tells us there exist non-abelian representations $\pi_1(S^3_r(K(11, 3)))\to\widetilde{PSL}(2, \mathbb{R})$ for each $r\in(-4, 0)\cup(0, 8).$ Since Hatcher and Thurston (\cite{HatcherThurston}) showed that all surgeries on non-torus 2-bridge knots are irreducible, $S^3_r(K(11, 3))$ is left-orderable for each $r\in(-4, 0)\cup(0, 8)$ by the Boyer-Rolfsen-Wiest Theorem.
\end{proof}

\section{An Infinite Family of 2-Bridge Knots}

We can extend the above result with some observations from Wang in \cite{Wang}. Recall that to each 2-bridge knot $K(p, q)$, there exists a diagrammatic representation of the knot associated to tangles given by the continued fraction expansion of $q/p.$ For the case of $6_2 = K(11, 3)$, the continued fraction of relevance is $[3, 1, 2].$ The following theorem tells us a certain construction of an infinite family of 2-bridge knots that is  group-theoretically relevant to us:

\begin{theorem}[\cite{Wang} Theorem 3.11, \cite{ORS} Theorem 6.1, Proposition 6.2]\label{nondegenerate}
Let $L$ be a 2-bridge knot with continued fraction representation $[a_1, \dots, a_n] := \mathbf{a}.$ Let $\mathbf{a}^{-1} = [-a_n, \dots, -a_1].$ Then, for integers $c_1, \dots, c_{2n} \in \mathbb{Z}$ and numbers $\epsilon_1, \dots, \epsilon_{2n+1}\in\{-1, 1\},$ consider the 2-bridge knot $K$ defined as the following continued fraction: $$[\epsilon_1\mathbf{a}, 2c_1, \epsilon_2\mathbf{a}^{-1}, 2c_2, \epsilon_3\mathbf{a}, 2c_3, \dots, \epsilon_{2n+1}\mathbf{a}^{\pm1}].$$ Then, there exists a homomorphism $\varphi:\pi_1(S^3 - K)\to\pi_1(S^3 - L)$ which preserves the peripheral subgroup, and when restricted to the peripheral subgroups is invertible. Moreover, there exist choices of meridians and longitudes $\mu_L, \mu_K, \lambda_L, \lambda_K,$ such that $\varphi(\mu_K) = \mu_L$ and $\varphi(\lambda_K) = \lambda_L^d$ where $$d = \epsilon_1 - \epsilon_2 +\epsilon_3 -\dots \pm \epsilon_{2n+1}.$$
\end{theorem}

\begin{remark}
Note in particular that $d\neq0$ since $d$ is odd. 
\end{remark}

What is particularly convenient to us is that this strange-looking condition on the knot complement groups is also associated to a left-orderability condition:

\begin{theorem}[from the proof of \cite{Wang} Proposition 3.2]\label{nondegeneratelo}
Suppose that $K, L\subset S^3$ are nontrivial knots and $\varphi:\pi_1(S^3 - K)\to\pi_1(S^3 - L)$ is a homomorphism preserving peripheral subgroups. Choosing meridians and longitudes $\mu_L, \mu_K, \lambda_L, \lambda_K,$ we may express $\varphi,$ restricted to the peripheral subgroups, as a matrix $\begin{pmatrix} a & b \\ c & d \\ \end{pmatrix}.$ Let $p/q\in\mathbb{Q}$ be a fraction in lowest common terms. Then, if $\pi_1(S^3_{ap/({bp+dq})}(L))$ is left-orderable and $S^3_{p/q}(K)$ is an irreducible 3-manifold, then $\pi_1(S^3_{p/q}(K))$ is left-orderable.
\end{theorem}

These lovely theorems tell us that we can construct an infinite family of 2-bridge knots on which we know a left-orderable interval of surgeries. Let $\mathbf{a} = [3,1,2]$ and $\mathbf{a}^{-1} = [-2, -1, -3].$ For integers $\mathbf{c} = (c_1, \dots, c_{2n}) \in \mathbb{Z}^{2n}$ and numbers $\boldsymbol{\epsilon} = \epsilon_1, \dots, \epsilon_{2n+1}\in\{-1, 1\}^{2n+1},$ let $K(\mathbf{c}, \boldsymbol{\epsilon})$ be the 2-bridge knot defined as the following continued fraction: $$[\epsilon_1\mathbf{a}, 2c_1, \epsilon_2\mathbf{a}^{-1}, 2c_2, \epsilon_3\mathbf{a}, 2c_3, \dots, \epsilon_{2n+1}\mathbf{a}^{\pm1}].$$ By the uniqueness of continued fraction representations,  $K(\mathbf{c}, \boldsymbol{\epsilon})$ will never be a torus knot (as torus knots are $K(2n+1, 1).$) Then, Theorem \ref{nondegenerate} tells us that there exists a homomorphism $\varphi: \pi_1(S^3 - K(\mathbf{c}, \boldsymbol{\epsilon}))\to \pi_1(S^3 - 6_2)$ which preserves the peripheral subgroup. Under suitable choices of meridians and longitudes, we may write the restriction of $\varphi$ to the peripheral subgroups as the matrix $\begin{pmatrix} 1 & 0 \\ 0 & d \\ \end{pmatrix}.$ Hence, Theorem \ref{nondegeneratelo} tells us that $\pi_1(S^3_{p/q}(K(\mathbf{c}, \boldsymbol{\epsilon})))$ is left-orderable whenever the manifold $S^3_{p/q}(K(\mathbf{c}, \boldsymbol{\epsilon}))$ is irreducible and when $p/dq$ is a left-orderable slope for the knot $6_2.$ We know that all surgeries on non-torus 2-bridge knots are irreducible by \cite{HatcherThurston}, and since $|d|\geq1,$ Theorem \ref{62} tells us that $p/dq$ is a left-orderable slope for $6_2$ whenever $p/q\in(-4, 4)$, thus we have the following corollary:

\begin{corollary}
Let $\mathbf{c} = (c_1, \dots, c_{2n}) \in \mathbb{Z}^{2n}$ and $\boldsymbol{\epsilon} = \epsilon_1, \dots, \epsilon_{2n+1}\in\{-1, 1\}^{2n+1}.$ Let $d = \epsilon_1 - \epsilon_2 +\epsilon_3 -\dots \pm \epsilon_{2n+1}.$ \begin{itemize}
	\item If $d>0,$ then $\pi_1(S^3_{p/q}(K(\mathbf{c}, \boldsymbol{\epsilon})))$ is left-orderable for all $p/q\in(-4d, 8d)\cap\mathbb{Q}.$
	\item If $d<0,$ then $\pi_1(S^3_{p/q}(K(\mathbf{c}, \boldsymbol{\epsilon})))$ is left-orderable for all $p/q\in(-8|d|, 4|d|)\cap\mathbb{Q}.$
\end{itemize}
In particular, $\pi_1(S^3_{p/q}(K(\mathbf{c}, \boldsymbol{\epsilon})))$ is left-orderable for all $p/q\in(-4, 4)\cap\mathbb{Q}.$
\end{corollary}

\section{Conclusion}

Theorems \ref{2bridge} and \ref{hyperbolic} proved above give some hope of showing left-orderability of the fundamental groups of many more manifolds obtained by surgery on 2-bridge knots. I have demonstrated here the power of these methods by extending the range of intervals particularly for the knot $6_2.$ This used ad hoc methods of estimating the solutions to the polynomials $P(t, u)$ and $Q(t, u)$ with the implicit function theorem and intermediate value theorem. However, I hope that these methods can be used alongside properties of the polynomials $P$ and $Q$ to extract more information in the general case.

This method is, however, lacking in two key areas. The first is that it provides little hope of extending beyond 2-bridge knots or other knots with similar  (for that matter, it is extremely unclear how the information about whether a 2-bridge knot is $L$-space enters into this entire cycle of ideas we have been discussing in most of the last two chapters.) 

The more dire problem is that this method relies on exhibiting a non-trivial $\widetilde{PSL}(2, \mathbb{R})$-representation to prove left-orderability. However, there is no guarantee that a left-orderable manifold admits a non-trivial representation of its fundamental group to $\widetilde{PSL}(2, \mathbb{R}).$ In fact, Gao has found an infinite family of integral homology spheres that are non-$L$-spaces in \cite{Gao}. Perhaps these spaces may provide counterexamples to the $L$-space conjecture; alternatively, they are left-orderable but the proof is more complicated than we would hope.

\bibliographystyle{abbrv} 
\bibliography{bib}

\end{document}